\documentclass[final]{dmtcs-episciences}
\usepackage[utf8]{inputenc}
\usepackage{subcaption}
\usepackage{amssymb,amsmath,amsthm,mathabx,verbatim}

\newtheorem{theorem}{Theorem}

\newtheorem{lemma}{Lemma}

\newcommand{\B}{\mathrm{B}}
\newcommand{\dRec}{\mathrm{dRec}}

\author{Jean Cardinal\affiliationmark{1}
  \and Vera Sacrist\'an\affiliationmark{2}
  \and Rodrigo I. Silveira\affiliationmark{2}}
\title[A note on flips in diagonal rectangulations]{A Note on Flips in Diagonal Rectangulations}
\affiliation{
  Universit\'e libre de Bruxelles (ULB), Belgium\\
  Universitat Polit\`ecnica de Catalunya (UPC), Spain}

\received{2018-2-26}

\revised{2018-10-4}

\accepted{2018-10-9}

\publicationdetails{20}{2018}{2}{14}{4315}

\begin{document}

\maketitle

\begin{abstract}
Rectangulations are partitions of a square into axis-aligned rectangles. 
A number of results provide bijections between combinatorial equivalence classes of rectangulations and families of pattern-avoiding permutations.
Other results deal with local changes involving a single edge of a rectangulation, referred to as flips, edge rotations, or edge pivoting.
Such operations induce a graph on equivalence classes of rectangulations, related to so-called {\em flip graphs} on triangulations and other
families of geometric partitions.
In this note, we consider a family of flip operations on the equivalence classes of {\em diagonal rectangulations}, and their interpretation as transpositions in the associated {\em Baxter} permutations, avoiding the vincular patterns $\{ 3\underline{14}2, 2\underline{41}3\}$. This complements results from Law and Reading (JCTA, 2012) and provides a complete characterization of flip operations on diagonal rectangulations, in both geometric and combinatorial terms.
\end{abstract}
\maketitle

\section{Introduction}

\subsection{Flip graphs}

The analysis of geometric partitions of space, such as triangle meshes, binary space partitions, and floorplans for integrated circuits
plays a major role in discrete and computational geometry and its applications.
In order to understand the underlying combinatorial structure of these partitions, it is often useful to define elementary operations that modify this structure locally.
We can then connect distinct partitions using sequences of such operations.

In triangulations, such a notion is known under the term of \emph{flip}.
A flip in a triangulation is typically defined as the replacement of an edge shared by two triangles forming a convex quadrilateral by the other diagonal of the quadrilateral. This allows the definition of a \emph{flip graph}, the vertices of which are triangulations, and in which two triangulations are adjacent whenever one can be obtained from the other by a single flip. This is a special case of the more general notion of \emph{reconfiguration graph}.

Flip graphs have applications in enumeration and random generation of geometric partitions, as well as optimization. 
The notion of flip graph has been studied for many distinct families of triangulations (maximal planar graphs and triangulations of a point set~\cite{BH09,BV11}, triangulations of a topological surface~\cite{N94}, see also \cite{DRS10} and references therein), and generalized to other families of geometric partitions, such as domino tilings~\cite{R04}, quadrangulations~\cite{N96a}, and rectangulations, the topic of the present contribution.
Flip graphs have been shown to have intimate links with many important structures in combinatorics, such as the Catalan objects~\cite{Catalan}, the Tamari lattice and the associahedra~\cite{Tamari}, cyclohedra~\cite{S03}, and partial cubes~\cite{E10}.

\subsection{Geometric partitions and pattern-avoiding permutations}

There exists a collection of results establishing bijections between families of geometric space partitions and pattern-avoiding permutations. We will use the \emph{word notation} for permutations, in which a permutation $\sigma$ is denoted by the word $\sigma (1) \sigma (2)\ldots \sigma (n)$. A permutation $\sigma$ is said to contain the \emph{pattern} $\pi$, where $\pi$ is another permutation, whenever there exists a subsequence of $\sigma$ whose elements are in the same relative order as the elements of $\pi$. Hence for instance the permutation $45213$ contains the pattern $213$, and also the pattern $3412$ in the form of the subsequence $4513$. Pattern-avoiding permutation classes are families of permutations that do not contain any occurence of one or more given patterns.

It is well known, for instance, that triangulations of a convex $(n+2)$-gon are in one-to-one correspondence with $312$-avoiding permutations on $n$ elements, and those are counted by the Catalan numbers (OEIS\footnote{Online Encyclopedia of Integer Sequences: \url{https://oeis.org/}} A000108). Similarly, guillotine partitions of a rectangle into $n$ rectangles, obtained by recursive splitting with a horizontal or vertical cut, can easily be seen to be in one-to-one correspondence with $\{3142, 2413\}$-avoiding permutations, called \emph{separable} permutations~\cite{BBL98}, which are counted by the Schr\"oder numbers (OEIS A006318). This has recently been generalized to separable \emph{$d$-permutations} and higher-dimensional guillotine partitions~\cite{AM10}.

We will use the \emph{underline} notation for more complex forbidden patterns in permutations, known as {\em vincular patterns}.
In this notation, an underlined block of elements indicates that they need to occur consecutively in the permutation.
For instance, forbidding the pattern $3\underline{14}2$ amounts to forbidding all occurrences of the pattern $3142$ with the added condition that $1$ and $4$ must occur consecutively. See Kitaev~\cite{K11} for precise definitions and further terminology. 

The objects of interest in this paper are \emph{rectangulations}, defined as partitions of a square into axis-aligned rectangles.
Several combinatorial equivalence classes of such rectangulations have been defined, which are known to be in one-to-one correspondence with families of permutations avoiding certain vincular patterns.
\emph{Mosaic floorplans}, for instance, have been shown to be in correspondence with \emph{Baxter permutations}, avoiding the patterns $3\underline{14}2$ and $2\underline{41}3$. This bijection seems to go back to the work of Dulucq and Guibert on Baxter permutations involving {\em twin binary trees}~\cite{DG96}. A complete description in terms of twin binary trees is given in Section 6 of Felsner et al.~\cite{FFNO11}. A simple description of a bijection for mosaic floorplans is given by Ackerman et al.~\cite{ABP06}.
Mosaic floorplans are also in bijection with \emph{twisted Baxter permutations} avoiding the patterns $3\underline{41}2$ and $2\underline{41}3$~\cite{LR12}. These two families are therefore in one-to-one correspondence, together with the family of $\{3\underline{14}2, 2\underline{14}3\}$-avoiding permutations. They will be instrumental in what follows.
Interestingly, the $\{ 3\underline{41}2, 2\underline{14}3 \}$-avoiding permutations are {\em not} in one-to-one correspondence with (twisted) Baxter permutations. These were actually shown by Asinowski et al. to count other equivalence classes of rectangulations~\cite{ABBMP13}, namely those preserving the neighborhood relation between the {\em segments} of the rectangulation.
{\em Generic} rectangulations, also known as {\em rectangular drawings}, are combinatorial rectangulations in which we take into account the adjacency relation between the rectangles.
Generic rectangulations have been shown to be in one-to-one correspondence with \emph{2-clumped} permutations, avoiding the vincular patterns $\{3\underline{51}24, 3\underline{51}42, 24\underline{51}3, 42\underline{51}3\}$~\cite{R12}. Table~\ref{tab:bijections} lists the known bijections between families of pattern-avoiding permutations and rectangulations.

\begin{table}
\begin{center}
  \begin{tabular}{|c|c|}
\hline 
Permutations & Rectangulations \\
\hline
\hline
Separable: $\{3142, 2413\}$-avoiding  & Slicing floorplans, \\
& or guillotine partitions \\
\hline
Baxter:
$\{ 3\underline{14}2, 2\underline{41}3\}$-avoiding & Mosaic floorplans, \\
Twisted Baxter:
$\{3\underline{41}2, 2\underline{41}3\}$-avoiding & or diagonal rectangulations, \\
$\{3\underline{14}2, 2\underline{14}3\}$-avoiding & or R-equivalent rectangulations \cite{Y03,ABP06,LR12} \\
\hline
$\{3\underline{41}2, 2\underline{14}3\}$-avoiding & S-equivalent rectangulations \cite{ABBMP13} \\ 
\hline
2-clumped: & generic rectangulations, \\ 
$\{3\underline{51}24, 3\underline{51}42,
24\underline{51}3, 42\underline{51}3\}$-avoiding & or rectangular drawings \cite{R12} \\
\hline
Separable $d$-permutations & Guillotine partitions \\
 & of $2^{d-1}$-dimensional boxes \cite{AM10}\\
\hline 
  \end{tabular}
  \end{center}
\caption{\label{tab:bijections}Known bijections between families of pattern-avoiding permutations and rectangulations.}
\end{table}

\subsection{Flips in rectangulations}

Different types of local operations can be defined on rectangulations, which have been given different names, such as flips, local move, edge rotations, or edge pivoting. In general, they all consist of replacing a horizontal edge of the rectangulation by a vertical one, or vice versa.
In what follows, and with a slight abuse of terminology, we will refer to all those under the common name of {\em flip}.

Law and Reading~\cite{LR12} described a family of flips on rectangulations and provided an elegant combinatorial characterization. They showed that two rectangulations were connected by such a flip if and only if they were in the cover relation of a certain natural lattice structure, analogous to the Tamari lattice on triangulations (hence on $312$-avoiding permutations), and part of the family of {\em Cambrian} lattices~\cite{Reading2012}. This lattice was also studied by Giraudo~\cite{G12} under the name of {\em Baxter lattice}.
Wide-reaching generalizations of these structures have been studied from the order-theoretic, algebraic, and polyhedral points of views by Reading~\cite{Reading2016}, Chatel and Pilaud~\cite{CP15}, and Pilaud and Santos~\cite{PS17}, among others.

Ackerman, Barequet and Pinter~\cite{ABP06b} defined related flip operations on rectangulations of a point set. 
These rectangulations are defined on a given point set so that every point lies on a segment of the rectangulation, and vice versa.
Ackerman et al. studied the flip graph induced by these operations~\cite{AABLMST16}.
The flips considered by Ackerman et al. are the same as the ones in Law and Reading whenever the point set lies on the diagonal.
Their results include a linear upper bound on the diameter of this flip graph (see \cite{AABLMST16}, Section 4).

An interesting application of flips in rectangulation to visualization of hierarchical data has been recently proposed by Sondag, Speckmann, and Verbeek~\cite{SSV18}. They use flips to maintain {\em treemap} layouts, representing information items by nested rectangles, under changes in the data. This improves the stability of the representation over time, and provably allows for all possible treemap layouts.

\subsection{Contribution}

We first describe a known bijection from diagonal rectangulations to Baxter permutations.
Then we consider flip operations on diagonal rectangulations, classify the different kinds of flips and give a combinatorial interpretation for each. Some of them, namely those involving edges that do not intersect the diagonal of the square, have already been characterized by Law and Reading~\cite{LR12}. We recall this characterization. For the others, we prove that the obtained flip graph is isomorphic to the graph on the corresponding Baxter permutations in which two Baxter permutations are adjacent whenever they differ by a single transposition of consecutive elements. 
We comment on the symmetry of the two interpretations. This provides a complete one-to-one correspondence not only between rectangulations and Baxter permutations, but also between these sets of natural operations on the geometric and combinatorial structures. Overall, this yields a complete characterization of flip operations in diagonal rectangulations. Illustrations of flip operations on rectangulations with three rectangles is given in Figure~\ref{fig:smallflipgraph}. 

\begin{figure}
\begin{center}
\includegraphics[width=.2\textwidth]{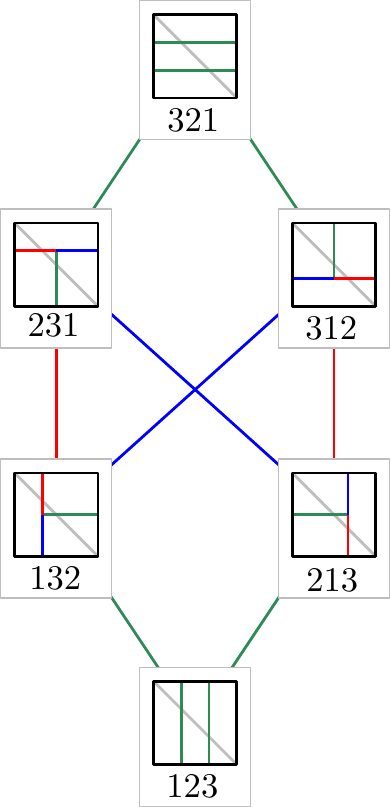}
\end{center}
\caption{\label{fig:smallflipgraph}Flips in diagonal rectangulations with three rectangles, together with their associated Baxter permutation.}
\end{figure}

\subsection{Plan}
In Section~\ref{sec:baxter} we provide some basic definitions and give a simple known bijection between diagonal rectangulations and Baxter permutations. 
In Section~\ref{sec:flips} we define and categorize a number of flip operations on diagonal rectangulations. 
Finally, in Section~\ref{sec:combinatorial} we give combinatorial characterizations for all the described flip operations. We first summarize the Law-Reading characterization in terms of the lattice structure (\ref{subsec:LR}), then proceed with the characterization of other flips (\ref{subsec:BCN}), which is our main new result.

\section{Diagonal rectangulations and Baxter permutations}
\label{sec:baxter}

In this section, we first define the combinatorial notion of diagonal rectangulation. 
Then we present maps from the set of diagonal rectangulations with $n$ rectangles to the set of permutations on $n$ elements. 
Those maps were described previously, and are known to be bijections between diagonal rectangulations and permutations avoiding some vincular patterns on four elements. 
They will be instrumental in the combinatorial interpretation of the flip graph on diagonal rectangulations.
The material of this section is adapted from Ackerman et al.~\cite{ABP06}, and Law and Reading~\cite{LR12}.
A description of an essentially equivalent map in terms of pairs of twin binary trees was given by Felsner et al.~\cite{FFNO11}.

\subsection{Diagonal rectangulations}

A rectangulation is a partition of the unit square into axis-aligned rectangles. 
We define {\em vertices} as corners of the rectangles, and {\em edges} as line segment connecting two vertices, with no other vertex in between.
The term {\em segment} is used to refer to inclusion-wise maximal line segments of the rectangulation, possibly composed of several edges. 
We consider only rectangulations in which every vertex has exactly three incident edges, except the four vertices of the square, which have exactly two incident edges.
We refer to the number of incident edges as the {\em degree} of the vertex, and classify the vertices into four self-explanatory classes depending on the orientation of their three incident edges and denoted by $\vdash$, $\dashv$, $\top$, and $\bot$.

We refer to the top-left to bottom-right diagonal of the square as the {\em main} diagonal, or simply the diagonal, when there is no ambiguity.
A {\em diagonal} rectangulation is a rectangulation in which every rectangle intersects the main diagonal.
However, since we deal with combinatorial structures, we actually define diagonal rectangulations as equivalence classes of such partitions of the square, with respect to moves that do not change the adjacency relation between the rectangles. 
Hence we allow changes in the positions of the vertices and edges, but we forbid moves that change the order of the vertices along a segment.
This allows the representation of any diagonal rectangulation in a unique way such that the intersections of the main diagonal with the segments are equidistributed. 
An example of diagonal rectangulation is given in Figure~\ref{fig:diag}.
The representation as a pair of {\em twin binary trees} is shown on Figure~\ref{fig:twin}.
We have the following characterization of (the equivalence classes of) diagonal rectangulations.

\begin{lemma}
\label{lem:diagchar}
A rectangulation is diagonal if and only if it does not contain one of the two forbidden configurations of Figure~\ref{fig:forbdiag}.
\end{lemma}

\begin{figure}\begin{center}
\includegraphics[page=1,scale=.6]{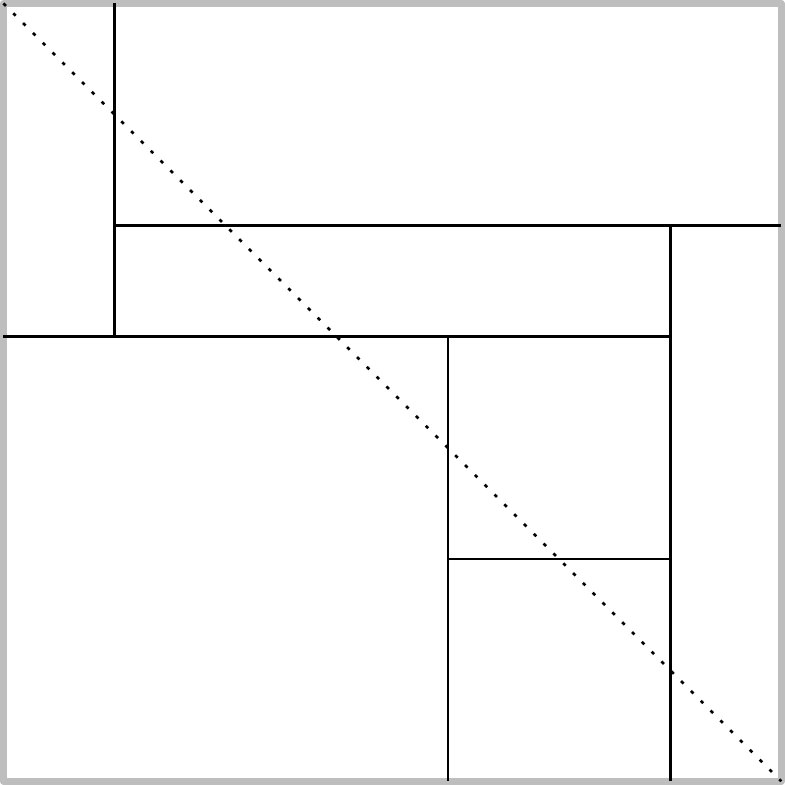}
\caption{\label{fig:diag}An example of diagonal rectangulation.}
\end{center}\end{figure}

\begin{figure}\begin{center}
\includegraphics[page=20,scale=.6]{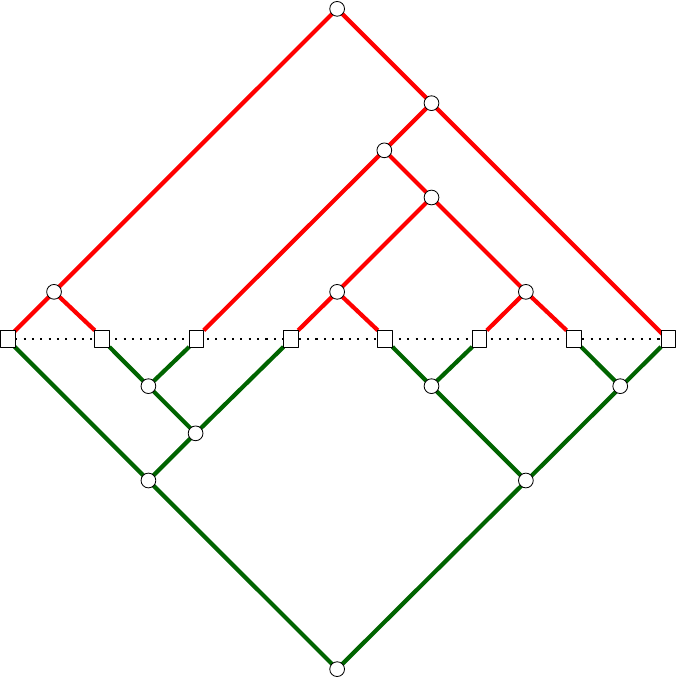}
\caption{\label{fig:twin}Twin binary trees associated with a diagonal rectangulation.}
\end{center}\end{figure}

\begin{figure}\begin{center}
\includegraphics[page=11,scale=.45]{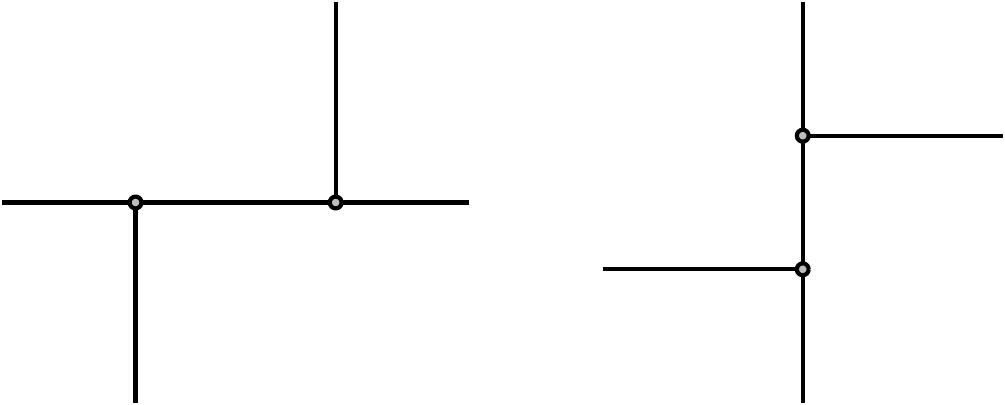}
\caption{\label{fig:forbdiag}Forbidden configurations in a diagonal rectangulation.}
\end{center}\end{figure}

The lemma can be deduced from the techniques in~\cite{ABP06}. We provide a brief sketch of an alternative proof. The necessity of the condition is clear. To prove sufficiency, one can greedily find a drawing in which all rectangles intersect a monotone curve in the square,
from the top left to the bottom right corner. This splits the drawing into two binary trees, which can be redrawn so that the leaves are equidistributed on the main diagonal.

We can also consider the equivalence classes of rectangulations for which we can change the relative position of vertices along a segment.
Two rectangulations are said to be equivalent when one can be obtained from the other by performing so-called {\em wall slides}, as shown on Figure~\ref{fig:wallslides}. Wall slides modify the adjacency relation among the rectangles, by exchanging the order of two vertices along a segment.
The equivalence relation is sometimes referred to as {\em $R$-equivalence}~\cite{ABBMP13}.
The $R$-equivalence classes are called {\em mosaic floorplans}.

\begin{lemma}
\label{lem:diagrep}
Every mosaic floorplan, or $R$-equivalence class, has a unique representative as a diagonal rectangulation. 
\end{lemma}

\begin{figure}\begin{center}
\includegraphics[page=10,scale=.45]{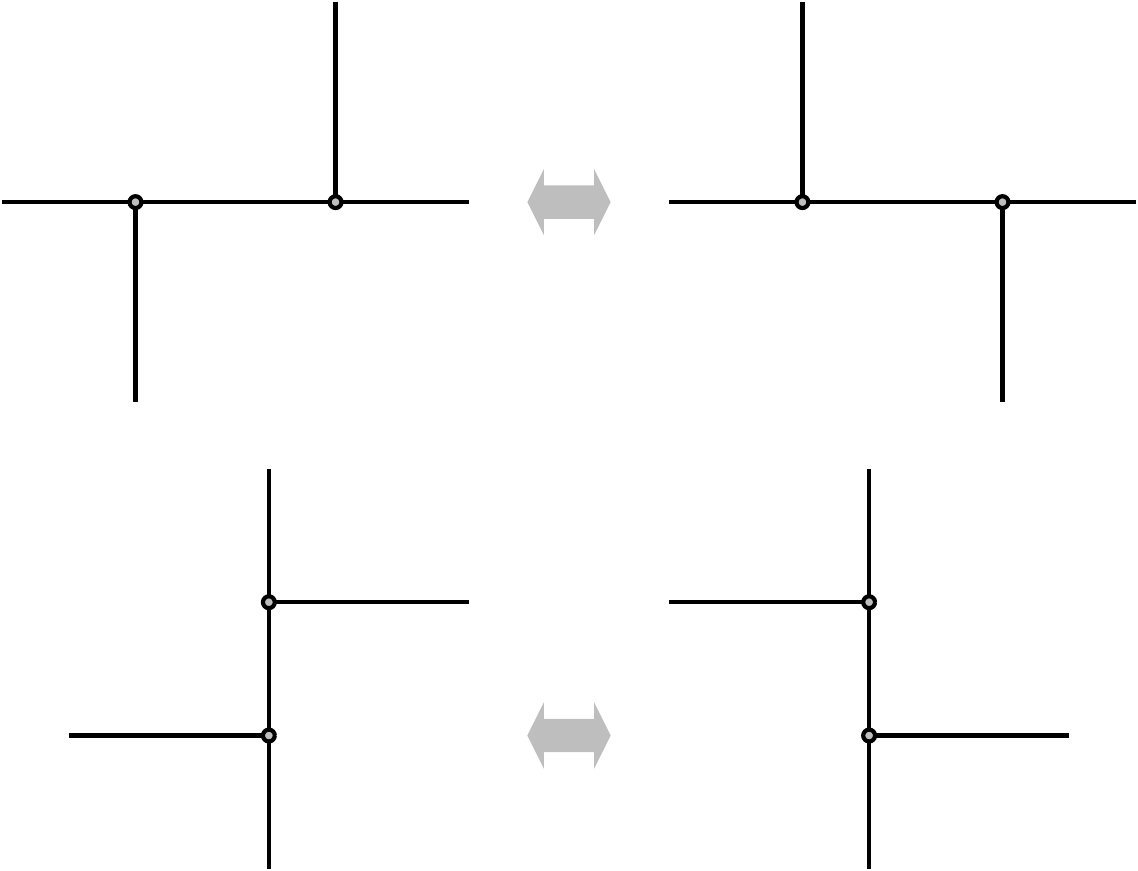}
\caption{\label{fig:wallslides}Wall slides.}
\end{center}\end{figure}

\subsection{A map from permutations to diagonal rectangulations}

Before delving into the details of the bijections, we first describe a map $\rho$ from any permutation to a diagonal rectangulation.

Given a permutation $\pi$ on $n$ elements, we consider the square to be dissected and divide its main diagonal into $n$ intervals, that we label
$1,2,\ldots ,n$, from left to right.
We then proceed iteratively by adding rectangles intersecting the intervals $\pi_1, \pi_2, \ldots ,\pi_n$. 
Before the $i$th step, we denote by $T_{i-1}$ the union of the rectangles that have already been drawn, together with the left and bottom edges of
the square. 
If $\pi_i = j$, we draw a rectangle intersecting the $j$th interval of the diagonal. We consider the left endpoint $\ell$ of this interval.
If $\ell$ is on the boundary of $T_{i-1}$, then the upper left corner of the $i$th rectangle is the highest point above $\ell$ that belongs to $T_{i-1}$. Otherwise, the upper left corner is the rightmost point on the boundary of $T_{i-1}$ that is directly left of $\ell$.
Similarly, consider the right endpoint $r$ of the $j$th interval. If $r$ is on the boundary of $T_{i-1}$, then the lower right corner of the new rectangle is the rightmost point of $T_{i-1}$ that is directly right of $r$. Otherwise, it is the highest point on the boundary of $T_{i-1}$ lying directly below $r$. The map is illustrated on Figure~\ref{fig:map}.

\begin{figure}\begin{center}
\includegraphics[page=14,scale=1]{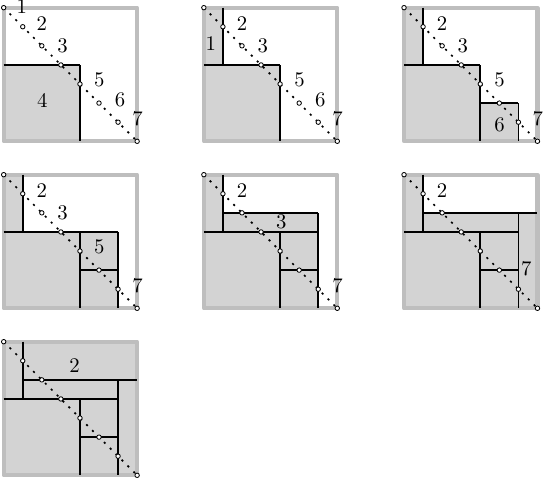}
\caption{\label{fig:map}The map $\rho$ from permutations to diagonal rectangulations. Construction of the diagonal rectangulation $\rho(4165372)$.}
\end{center}\end{figure}

\subsection{Maps from diagonal rectangulations to permutations}

Given the algorithm above, one can realize that many distinct permutations can yield the same rectangulation.
For instance, one can check on the example of Figure~\ref{fig:map} that rectangles 6 and 5 can both be drawn before rectangle 1, hence that $\rho (4165372) = \rho(4651372)$.
In general, given a diagonal rectangulation $R$, one can easily find a permutation $\pi$ such that $\rho(\pi)=R$ by moving backwards in the order in which the rectangles
are drawn. At each step, there always exists a rectangle in the rectangulation that is drawn correctly when the above procedure is applied.
In order to define a map from rectangulations to permutations, we need a tie-breaking rule, which allows to decide univocally which is the next rectangle to pick, and therefore to choose one well-defined element from each preimage $\rho^{-1}(R)$.
There are two simple rules we can apply: the leftmost and rightmost rules. In those rules, the next rectangle we pick is the one that is leftmost (respectively rightmost) on the diagonal. In all cases, the rectangles of the given rectangulation are first labeled in what we call the $\boxbackslash$-order, defined as the order in which they intersect the diagonal.

The preimage $\rho^{-1}(R)$ can in fact be interpreted as the set of common linear
extensions of the two binary trees corresponding to $R$: the tree above the main diagonal
oriented towards its root and the tree below the main diagonal oriented from its
root. In particular, the three tie-breaking rules are ways to chose between these
linear extensions. 

Let us first consider the leftmost rule, as illustrated on Figure~\ref{fig:bijtwB}. One can check that the permutation that is produced using the leftmost rule avoids the vincular patterns $3\underline{41}2$ and $2\underline{41}3$.
Indeed, in both cases, the leftmost rule prescribes that rectangle 1 is chosen before rectangle 4.
Permutations avoiding those two patterns are known as \emph{twisted Baxter} permutations.
In fact, only twisted Baxter permutations can be produced, and the map is known to be a bijection between diagonal rectangulations and twisted Baxter permutations~\cite{LR12}.

\begin{figure}\begin{center}
\includegraphics[page=9,scale=.6]{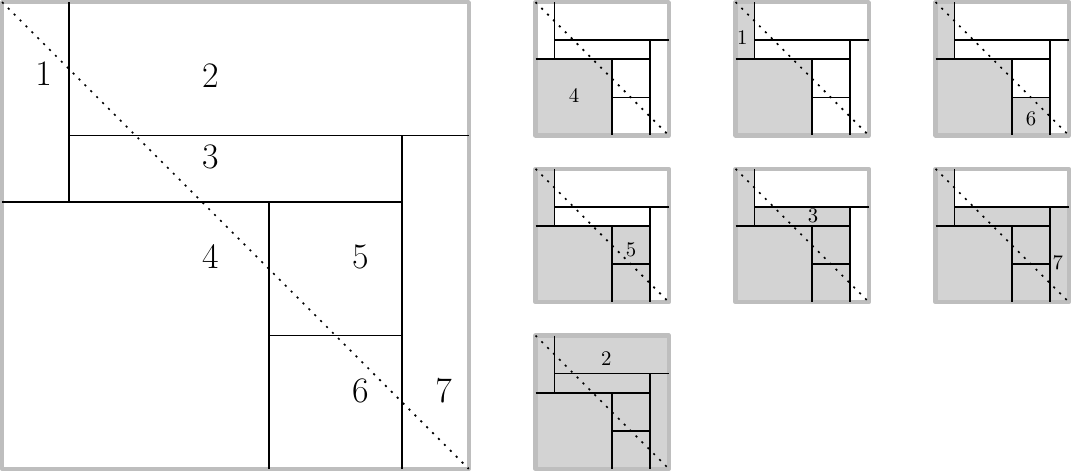}
\caption{\label{fig:bijtwB}Illustration of the map given by the leftmost order on the rectangulation $R$ of Figure~\ref{fig:diag}: the labels of the rectangles are given by the $\boxbackslash$-order (left), and listed in the leftmost order (right). The obtained twisted Baxter permutation is $4165372$.}
\end{center}\end{figure}

Similarly, the rightmost rule yields a bijection between diagonal rectangulations and $\{3\underline{14}2,2\underline{14}3\}$-avoiding permutations.
In fact, it can be shown that the preimage $\rho^{-1}(R)$ forms an interval, and the leftmost and rightmost rules select the bottom and top elements of these intervals respectively.

We now describe the tie-breaking rule that allows to define a bijection $\B$ between diagonal rectangulations and \emph{Baxter} permutations, which avoid the patterns $\{3\underline{14}2,2\underline{41}3\}$.
In order to define $\B$, we define another linear order on the rectangles of a rectangulation: the $\boxslash$-order. 
The $\boxslash$-order is obtained by taking the representative $\boxslash R$ of $R$ in the equivalence class of mosaic floorplans such that the bottom-left to top-right diagonal intersects every rectangle. 
From Lemma~\ref{lem:diagrep}, this representative exists and is unique.
The $\boxslash$-order is then simply the order in which this diagonal intersects the rectangles.
The Baxter permutation $\B(R)$ corresponding to a diagonal rectangulation $R$ is the order of the rectangle labels in the $\boxslash$-order, see Figure~\ref{fig:bijB}.
The map $B$ can then be described concisely as follows:
\begin{enumerate}
\item label the rectangles with respect to the $\boxbackslash$-order,
  \item enumerate the labels of the rectangles in the $\boxslash$-order.
  \end{enumerate}
The $\boxslash$-order is distinct from both the leftmost and the rightmost order, and can be shown to avoid the Baxter patterns.
An example illustrating this distinction is given on Figure~\ref{fig:orders}.
The following result is due to Ackerman et al.~\cite{ABP06}. In their proof, the description of the $\boxslash$-order involves \emph{block deletion} operations, but can be seen to be equivalent to ours.
\begin{theorem}
  The map $\B$ is a bijection between diagonal rectangulations with $n$ rectangles and Baxter permutations on $n$ elements.
  \end{theorem}

\begin{figure}\begin{center}
\includegraphics[page=2,scale=.6]{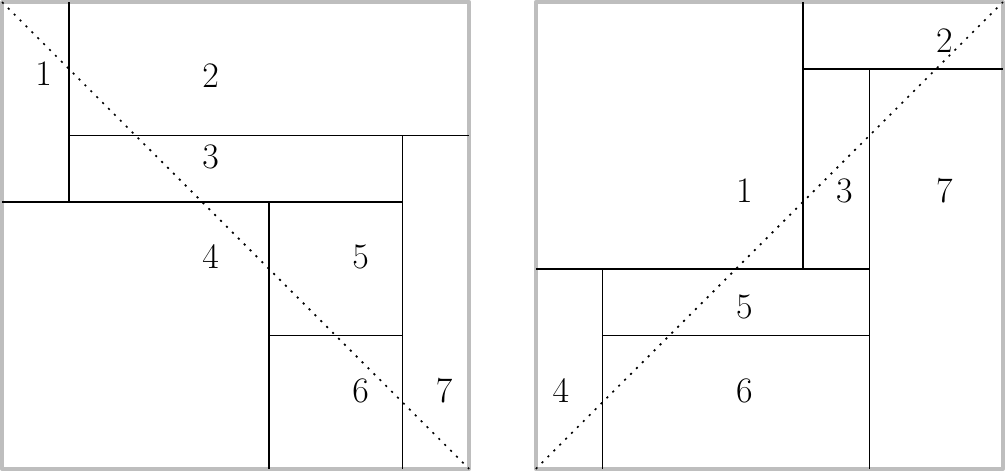}
\caption{\label{fig:bijB}Illustration of the map $\B$ on the rectangulation $R$ of Figure~\ref{fig:diag}: the labels of the rectangles are given by the $\boxbackslash$-order (left), and listed in the $\boxslash$-order (right). The rectangulation on the right is in the same equivalence class of mosaic floorplans as the one on the left. The obtained Baxter permutation is $\B(R)=4651372$.}
\end{center}\end{figure}

\begin{figure}\begin{center}
\includegraphics[page=13,scale=.6]{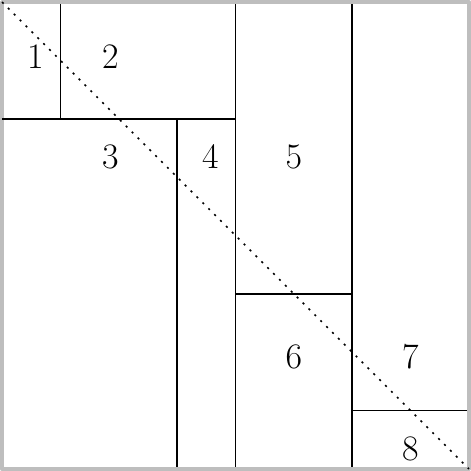}
\caption{\label{fig:orders}Example of diagonal rectangulation in which the leftmost, rightmost, and $\boxslash$-orders are all distinct. The leftmost order, hence the twisted Baxter permutation is $31426587$, the rightmost order is $34126857$, and the $\boxslash$-order, hence the Baxter permutation, is $34126587$.}
\end{center}\end{figure}

\subsection{Inversion}
We now give a relation between rectangulations produced by a Baxter permutation $\pi$ and its inverse $\pi^{-1}$. 
Note that inverses of Baxter permutations are also Baxter permutations.
The map $\rho$ from permutations of $n$ elements to diagonal rectangulations with $n$ rectangles is defined by iteratively drawing the rectangle given by the next element of the sequence on the main diagonal. We define a similar map $\rho'$ that produces a rectangulation in which the rectangles intersect the other, bottom-left to top-right, diagonal. The map $\rho'$ is simply defined as the composition of $\rho$ with a reflection with respect to the horizontal axis.
We observe that applying this map to the inverse permutation $\pi^{-1}$ yields the alternative diagonal representation of $\rho(\pi)$.

\begin{lemma}
\label{lem:inversion}
Let $\pi = B(R)$ for a rectangulation $R$. Then $\rho'(\pi^{-1}) = \boxslash R$.
\end{lemma}
\begin{proof}
  Consider the map $B'$ from rectangulations to permutations defined as follows:
\begin{enumerate}
\item label the rectangles with respect to the $\boxslash$-order,
  \item enumerate the labels of the rectangles in the $\boxbackslash$-order.
  \end{enumerate}
Note that this matches the description of the map $B$, except that we exchanged the roles of the two orders.
Since the roles of the indices and the elements are now exchanged, we must have that $B'(R)=\pi^{-1}$.
Now applying $\rho'$ on the permutation $\pi^{-1}$ amounts to inserting the rectangles in the order given by $\pi^{-1}$
along the other, bottom-left to top-right diagonal.
But since the rectangles were labeled by $B'$ in the $\boxslash$-order, this must yield $\boxslash R$.
\end{proof}

\section{Flips}
\label{sec:flips}

In this section, we present a geometric notion of flips in diagonal rectangulations. 
We consider only flipping edges that are not part of the boundary of the square.
We say that an edge is {\em matched} at one of its endpoint whenever this endpoint is incident to another edge with the same (horizontal/vertical) orientation.

\begin{figure}
\begin{center}
  \begin{subfigure}{.7\textwidth}
\begin{center}
\includegraphics[page=3,scale=.5]{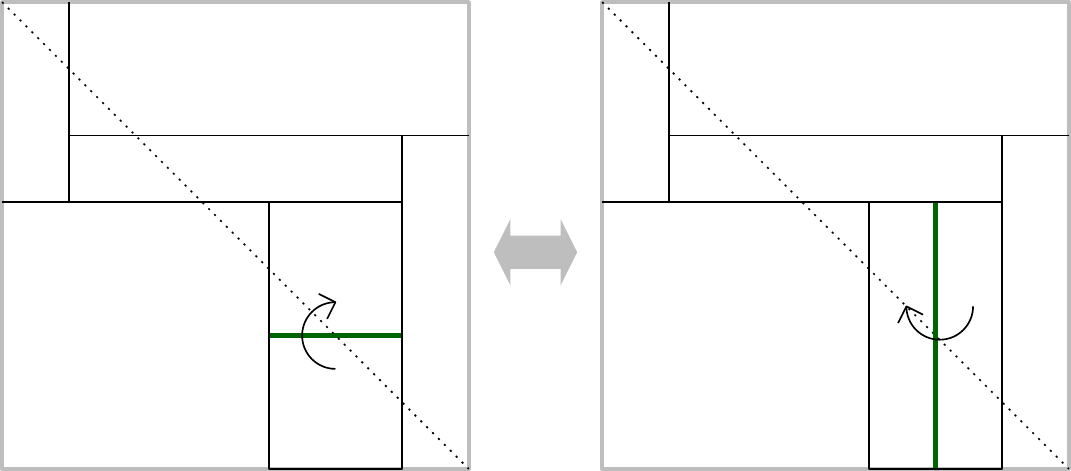}
\caption{\label{fig:triv}Simple flip.}
\end{center}
\end{subfigure}
\begin{subfigure}{.7\textwidth}
\begin{center}
\includegraphics[page=6,scale=.5]{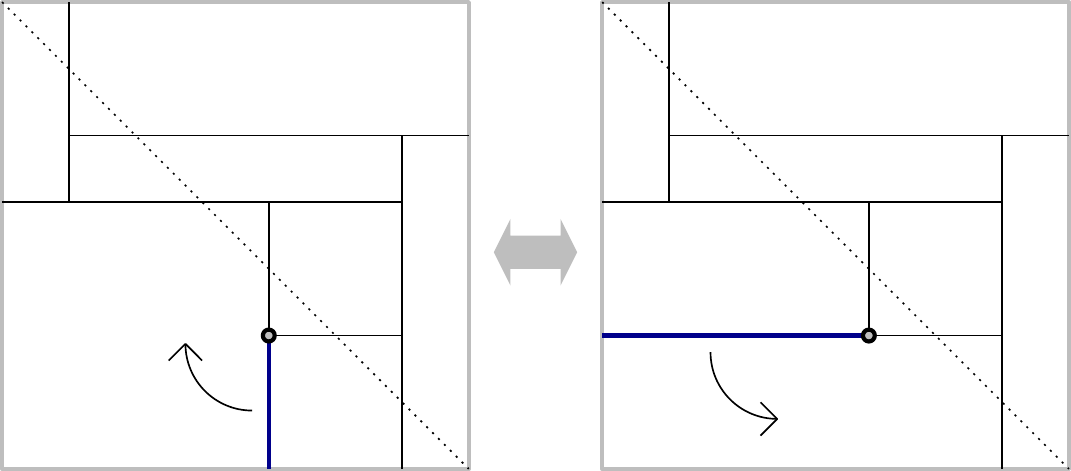}
\caption{\label{fig:flip1}Flip involving an edge that does not intersect the diagonal.}
\end{center}
\end{subfigure}
\begin{subfigure}{.7\textwidth}
\begin{center}
\includegraphics[page=7,scale=.5]{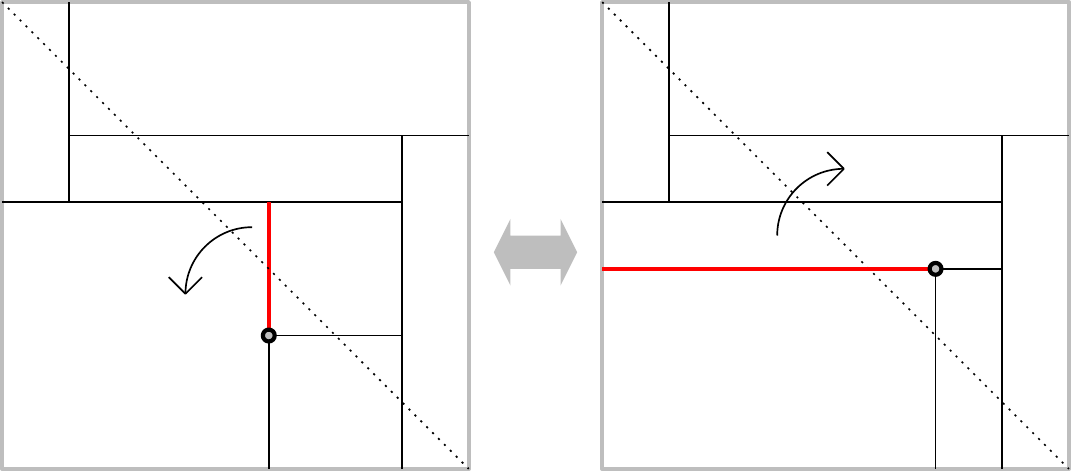}
\caption{\label{fig:flip2}Flip involving an edge intersecting the diagonal.}
\end{center}
\end{subfigure}
\end{center}
\caption{\label{fig:flips}The three kinds of flips in a diagonal rectangulation.}
\end{figure}

\subsection{Simple flips} 

Simple flips involve edges that separate two adjacent rectangles 
whose union is a rectangle itself. These are precisely the edges that are unmatched at both endpoints. 
All such edges must intersect the diagonal.
A simple flip consists in replacing such a horizontal edge by a vertical one, or vice versa.
When replacing the edge, we can always do it in such a way that the resulting rectangulation remains diagonal.
An example of simple flip in the rectangulation of Figure~\ref{fig:diag} is given in Figure~\ref{fig:triv}.

It is perhaps worth noting that flipping those edges is not sufficient to connect any pair of diagonal rectangulations. 
In other words, the simple flip graph is not connected.
Two rectangulations that differ only by simple flips have been called {\em S-equivalent} by Asinowski et al.~\cite{ABBMP13}, and the corresponding equivalence classes are shown to be in bijection with the $\{3\underline{41}2,2\underline{14}3\}$-avoiding permutations.

\subsection{Flips using rotations}

In some cases, an edge that is matched only at one of its endpoints can be \emph{rotated} around this endpoint to yield another diagonal rectangulation. Examples of such flips are given in Figures~\ref{fig:flip1} and~\ref{fig:flip2}. Flips of the kind given in Figure~\ref{fig:flip1} are exactly rotations in one of the twin binary trees.

Note that when such a flip involves an edge that intersects the main diagonal, like in Figure~\ref{fig:flip2}, merely replacing the rotated edge in a drawing of the rectangulation does not yield a drawing of the rectangulation that is diagonal, that is, some rectangles do not intersect the diagonal anymore. However, the rectangulation remains a proper diagonal rectangulation in the combinatorial sense, because no wall slide is needed to make all rectangles intersect the diagonal.

Together, all these flip operations define a flip graph on the set of diagonal rectangulations.
However, not all edges can be flipped.
An edge is said to be unflippable in two cases.

\begin{figure}
\centering
\begin{subfigure}{.5\textwidth}
  \centering
  \includegraphics[page=4,scale=.6]{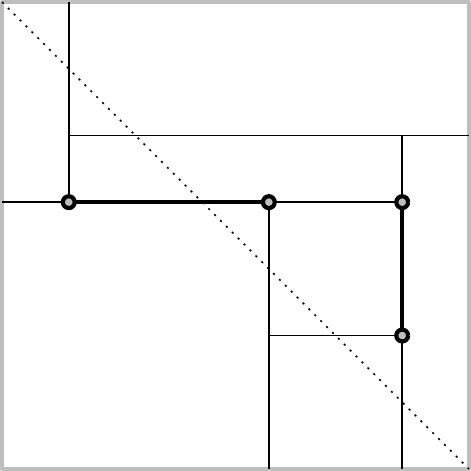}
  \caption{Unflippable edges matched at both endpoints.}
  \label{fig:unflip1}
\end{subfigure}%
\begin{subfigure}{.5\textwidth}
  \centering
  \includegraphics[page=5,scale=.6]{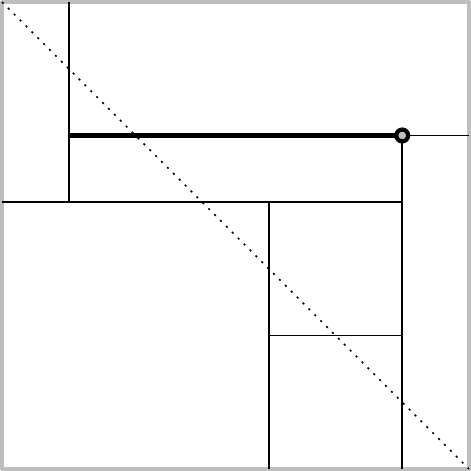}
  \caption{Unflippable edge matched at one endpoint.}
  \label{fig:unflip2}
\end{subfigure}
\caption{Unflippable edges.}
\label{fig:unflip}
\end{figure}

\subsubsection{Unflippable edges matched at both endpoints.}
If the edge is matched at both endpoints, rotating this edge around any of the two endpoints
yields a partition that is not a rectangulation. Examples are shown in Figure~\ref{fig:unflip1}.

\subsubsection{Unflippable edges matched at one endpoint.}
It can also be the case that an edge is matched at only one endpoint, and rotating it around this endpoint 
yields a rectangulation, but the obtained rectangulation is not diagonal. 
An illustration is given in Figure~\ref{fig:unflip2}.
We have the following lemma characterizing such unflippable edges.

\begin{lemma}
\label{lem:unflip2char}
Unflippable edges matched at only one endpoint must fall in one of the four types described in Figure~\ref{fig:unflip2b}.
Furthermore, all of them must intersect the diagonal.
\end{lemma}
\begin{proof}
By definition, flipping the edge must create one of the two configurations in Lemma~\ref{lem:diagchar}, shown in Figure~\ref{fig:forbdiag}.
Each of the two configurations can be forced to occur only after one of two edges have been rotated, hence can only  
happen in one of the four cases described.
The second statement can be proved by contradiction, by considering the four types of unflippable edges.
If, for one of them, the diagonal does not intersect the edge, then the rectangulation is not diagonal to start with.
\end{proof}

\begin{figure}\begin{center}
\includegraphics[page=12,scale=.45]{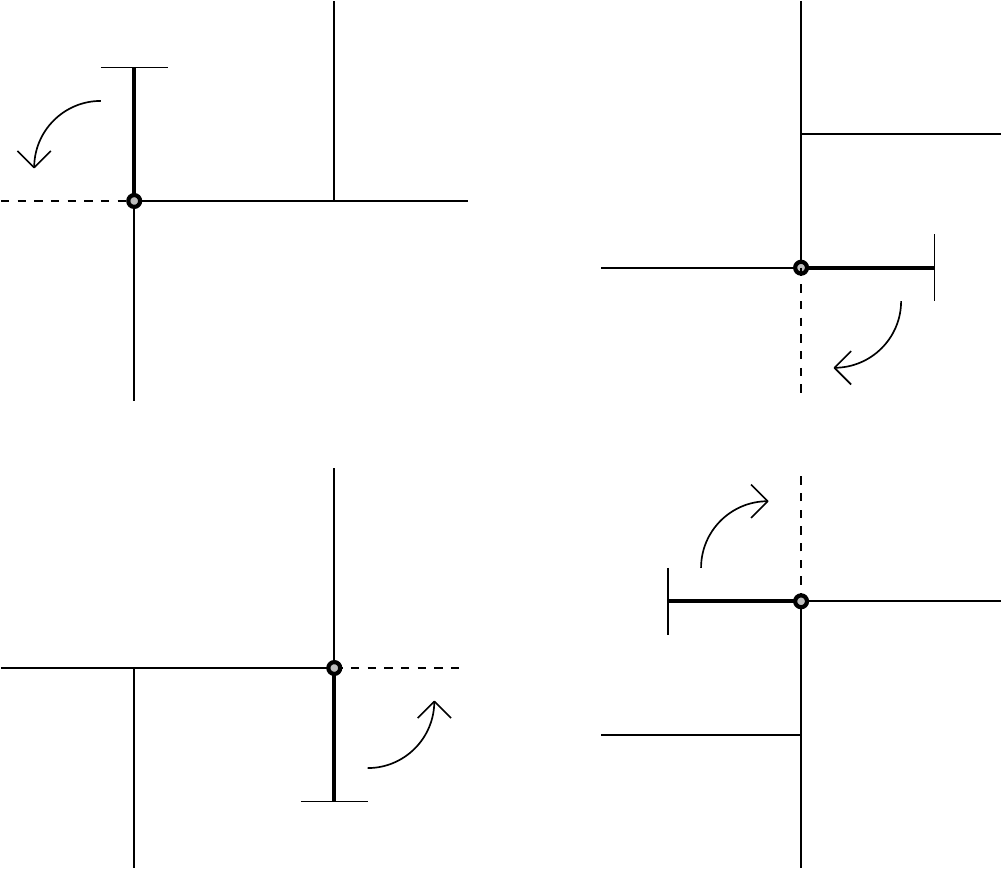}
\caption{\label{fig:unflip2b}The four types of unflippable edges matched at only one endpoint.}
\end{center}\end{figure}

All edges that are neither simply flippable nor unflippable according to the previous definitions can be flipped using a rotation to
get another diagonal rectangulation. 

\section{A complete combinatorial characterization of flips}
\label{sec:combinatorial}

In this section, we give a combinatorial characterization of all edges in the flip graph of diagonal rectangulations.
For this purpose, we will use some simple terminology on permutations.
A {\em transposition} maps a permutation $\pi = \pi (1) \pi (2)\ldots \pi(j) \ldots \pi(k)\ldots \pi (n)$ to a permutation $\pi'=\pi(1)\pi(2)\ldots \pi(k)\ldots \pi(j)\ldots \pi(n)$.
Furthermore, if the two values $j$ and $k$ satisfy $|\pi(j)- \pi(k)|=1$, then the transposition is said to be a {\em transposition of consecutive elements}.
If $k=j+1$, then the transposition is said to be an {\em adjacent transposition}.
Note that an adjacent transposition corresponds to a transposition of consecutive elements in the inverse permutation.

\subsection{Law-Reading flips}
\label{subsec:LR}

We first summarize a result of Law and Reading, characterizing some of the flip operations described above as a cover relation in a lattice,
which can be found in Section 7 of ~\cite{LR12}.
In what follows, we will use the term {\em Law-Reading flips} to refer to those flips.

In the original description, the set of Law-Reading flippable edges is constructed as follows: for every inner vertex, consider the two edges going towards (not necessarily intersecting) the diagonal. Consider the one that is matched and exclude it from the set. The remaining edges are Law-Reading flippable.
It can be checked that all Law-Reading flippable edges are flippable with respect to the definitions of Section~\ref{sec:flips}. The following lemma gives a simple alternative definition of Law-Reading flips.

\begin{lemma}
Law-Reading flips are exactly the flips that are either simple, or that involve the rotation of a flippable edge that does not intersect the diagonal, as illustrated in Figure~\ref{fig:flip1}.
\end{lemma}
\begin{proof}
Consider an edge that intersects the diagonal. If this edge is Law-Reading flippable, then it must be unmatched at both endpoints, since otherwise one of the endpoints would lock it. Hence it must be simply flippable. Conversely, suppose that this edge is flippable, but not simply flippable. Then it must be matched at exactly one endpoint. But for this endpoint, the edge is towards the diagonal and therefore must be locked. Hence Law-Reading flips of edges intersecting the diagonal are exactly the simple flips.

Consider now an edge that does not intersect the diagonal. We need to show that it is flippable if and only if it is Law-Reading flippable. Suppose it is flippable. Then it must be matched at exactly one endpoint. This endpoint must be the closest to the diagonal, for otherwise the rectangulation is not diagonal. But then it cannot be locked, and is Law-Reading flippable. On the other hand, suppose it is Law-Reading flippable. Then it cannot be locked, and can only be matched at the endpoint that is the closest from the diagonal. From Lemma~\ref{lem:unflip2char}, unflippable edges matched at one endpoint must intersect the diagonal. Therefore this edge must be flippable.
\end{proof}

We now give a combinatorial characterization of Law-Reading flips proved in~\cite{LR12} using the map from rectangulations to Baxter permutations. Before stating the result, we must define the lattice $\dRec_n$ of diagonal rectangulations with $n$ rectangles. 

\subsubsection{A lattice on diagonal rectangulations.}
The {\em weak order} (also known as the weak {\em Bruhat} order) is a partial order on the set $S_n$ of permutations of $n$ elements in which a permutation $\pi$ is smaller than another permutation $\pi'$ whenever the set of inversions of $\pi$ is a subset of the set of inversions of $\pi'$.
The cover relation of the weak order is the set of pairs of permutations that differ by a single adjacent transposition. The weak order is a classical, well-studied order, and known to be a lattice. 

The lattice $\dRec_n$ on diagonal rectangulations can be defined as the restriction of the weak order to the Baxter permutations corresponding to diagonal rectangulations with $n$ rectangles.
In fact, it can be shown that the preimages $\rho^{-1}(R)$ of $\rho$ form a lattice congruence on the weak order. The lattice $\dRec_n$ is the quotient of the weak order with respect to this congruence. Therefore, $\dRec_n$ may as well be defined by restricting the weak order to any set of representatives of each congruence class. More concretely, we can pick for any rectangulation $R$ any representative in $\rho^{-1}(R)$ and consider the order induced by those. For instance, the partial order induced by the weak order on the twisted Baxter permutations is isomorphic to $\dRec_n$. 

We can now state the connection between this order and Law-Reading flips in rectangulations. Recall that $B(R)$ is the Baxter permutation associated with the diagonal rectangulation $R$.
 \begin{theorem}[Law and Reading~\cite{LR12}]
\label{thm:LR}
Let $R$ and $R'$ be two diagonal rectangulations. Then $R$ and $R'$ are connected by a Law-Reading flip if and only if
$\B(R)$ and $\B(R')$ are in a cover relation in $\dRec_n$.
\end{theorem}

This means that the two Baxter permutations corresponding to the pair of rectangulations are related by a monotone sequence of adjacent transpositions, and the intermediate permutations, if any, are not Baxter permutations. Note that the Law-Reading flips have a simple interpretation in the representation of a rectangulation by twin binary trees. A Law-Reading flip then corresponds to a rotation in one of the two binary tree (see Section 5.3 in Giraudo~\cite{G12}).

\subsection{Barcelona flips}
\label{subsec:BCN}

We define Barcelona flips as those flips that involve a flippable edge intersecting the main diagonal. Barcelona flips are either simple flips, 
or flips involving the rotation of an edge intersecting the diagonal, as shown in Figure~\ref{fig:flip2}.

\subsubsection{Barcelona flips in $R$ and Law-Reading flips in $\boxslash R$.}

\begin{figure}\begin{center}
\includegraphics[page=3,scale=.6]{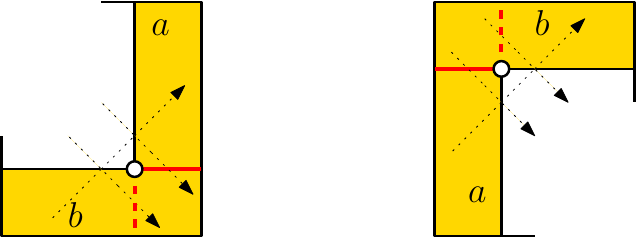}
\caption{\label{fig:proofa}Illustration of the proof of Lemma~\ref{lem:switch}. The $\boxbackslash$ and $\boxslash$-orders of the two rectangles $a$ and $b$ are indicated by the dotted arrows.}
\end{center}\end{figure}

\begin{lemma}
\label{lem:switch}
Let $R$ and $R'$ be two diagonal rectangulations that are connected by a Barcelona flip.
Then $\boxslash R$ and $\boxslash R'$ are connected by a Law-Reading flip.
\end{lemma}
\begin{proof}
We first consider the case where the Barcelona flip is a simple flip. Then the edge is unmatched a both endpoints, and remain so in $\boxslash R$. Hence $\boxslash R$ and $\boxslash R'$ are connected by a simple flip as well.

In the case where the Barcelona flip is not simple, it must involve two rectangles with labels $a$ and $b$ that can be in two possible distinct relative positions, as depicted in Figure~\ref{fig:proofa}.

We first remark that in the configuration on the left of Figure~\ref{fig:proofa} in $R$, the top left corners of $a$ and $b$ must respectively be $\top$ and $\vdash$ vertices. This is because otherwise the rectangulation after or before the flip contains one of the two forbidden configurations of Figure~\ref{fig:forbdiag} and cannot be diagonal. Similarly in the configuration on the right, the bottom right corners of $a$ and $b$ must respectively be $\bot$ and $\dashv$ vertices. Hence the relative position of the two rectangles $a$ and $b$ cannot be changed by wall slides, and remains the same in $\boxslash R$. This in turn implies that the other diagonal in $\boxslash R$ does not intersect the flipped edge, and that this edge is still matched at only one endpoint. Since it does not intersect the diagonal, Lemma~\ref{lem:unflip2char} implies that the edge is flippable in $\boxslash R$, and the flip is a Law-Reading flip. Applying the same reasoning starting with $R'$, we conclude that flipping this edge in $\boxslash R$ yields the rectangulation $\boxslash R'$.
\end{proof}

\begin{figure}\begin{center}
\includegraphics[page=15,scale=.6]{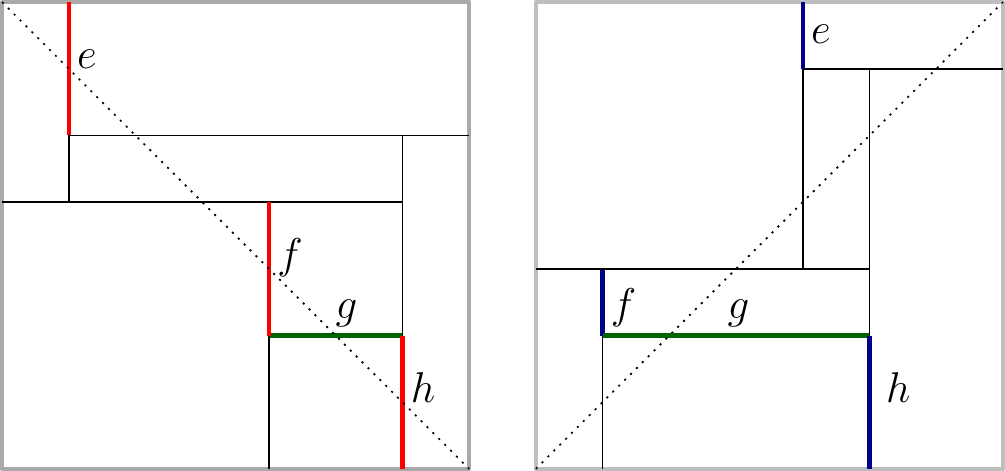}
\caption{\label{fig:BCNtoLR}Illustration of Lemma 6: edges $e,f,g,h$, that can be flipped by a Barcelona flip in $R$ (left) can be flipped by a Law-Reading flip in $\boxslash R$ (right).}
\end{center}\end{figure}

The Lemma is illustrated in Figure~\ref{fig:BCNtoLR}.
Combining the above lemma with Lemma~\ref{lem:inversion} on the way to obtain $\boxslash R$ from $B(R)$, and the characterization of Law-Reading flips in Theorem~\ref{thm:LR}, one can already conclude that a Barcelona flip in a rectangulation $R$ corresponds to a sequence of adjacent transpositions in the inverse permutation $B(R)^{-1}$. It is perhaps tempting to conjecture at this stage that Barcelona and Law-Reading flips are exactly dual to each other, in the sense that the set of Barcelona flips in $R$ is in bijection with the set of Law-Reading flips in $\boxslash R$. This is not the case. In what follows, we show that the Barcelona flips are in correspondence with the Law-Reading flips in $\boxslash R$ that involve a {\em single} adjacent transposition in $B(R)^{-1}$, that is, a single transposition of consecutive elements in $B(R)$. Law-Reading flips in $\boxslash R$ that involve more than one transpositions do not have a direct interpretation in terms of flips in $R$.

\subsubsection{A characterization of Barcelona flips.}

\begin{theorem}
\label{thm:main}
Let $R$ and $R'$ be two diagonal rectangulations. Then $R$ and $R'$ are connected by a Barcelona flip if and only if
$\B(R)$ and $\B(R')$ differ by a single transposition of consecutive elements.
\end{theorem}
\begin{proof}
$(\Rightarrow)$ First suppose that $R$ and $R'$ are connected by such a flip. If this is a simple flip, it is not difficult to verify, by referring to the descriptions of the map $B$, that the permutations indeed differ by a single transposition of consecutive elements.

Now suppose it is not a simple flip. From Lemma~\ref{lem:switch}, we have that $\boxslash R$ and $\boxslash R'$ are connected by a nonsimple Law-Reading flip. But those involve precisely the edges that do not intersect the diagonal, hence the $\boxslash$-order labels of the rectangles in $R$ and $R'$ are the same. It remains to observe that since flipping the edge does not create any obstruction to the rectangulation being diagonal, the $\boxbackslash$-order of the rectangles $a$ and $b$ (refer to Figure~\ref{fig:proofa}) is simply reversed. We conclude that the flip corresponds to the single transposition of the two elements $a$ and $b$ in $\B (R)$.

\begin{figure}
\begin{subfigure}{.65\textwidth}
  \centering
\includegraphics[page=1,scale=.5]{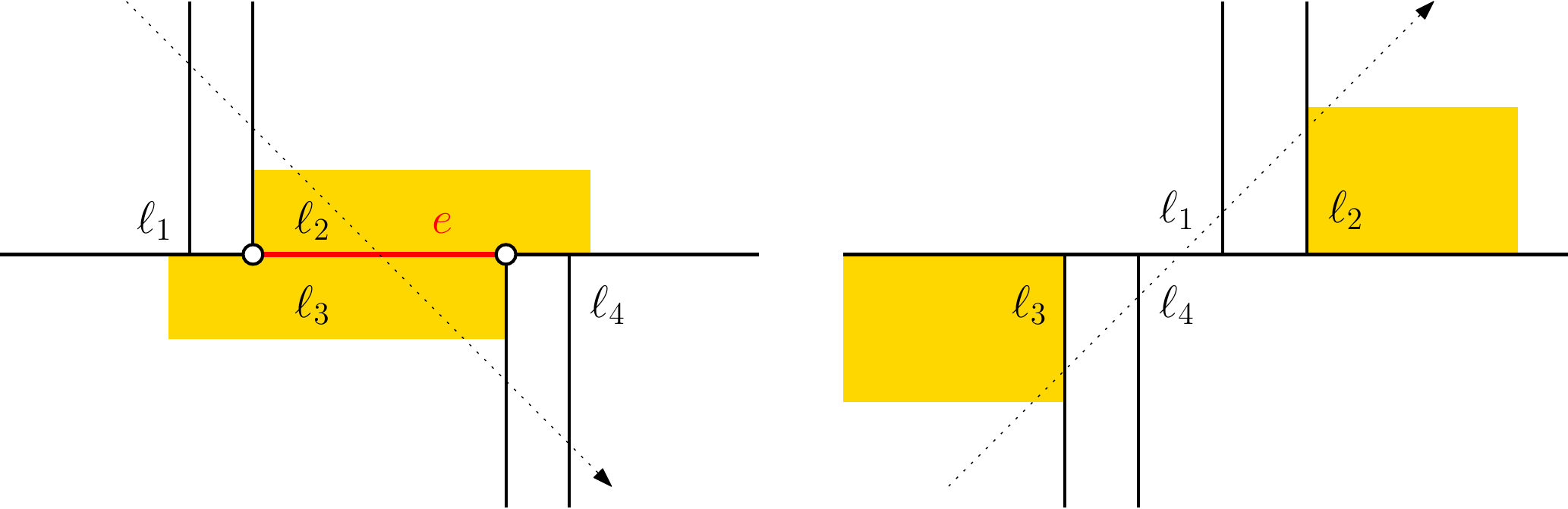}
\caption{\label{fig:proof1}Unflippable edge matched at both endpoints.}
\end{subfigure}
\hspace{.5cm}
\begin{subfigure}{.3\textwidth}
\centering
\includegraphics[page=2,scale=.5]{proof.pdf}
\caption{\label{fig:proof2}Unflippable edge matched at one endpoint.}
\end{subfigure}
\caption{Illustrations for the proof of Theorem~\ref{thm:main}.}
\end{figure}

$(\Leftarrow)$ We now suppose that the two Baxter permutations $\pi=\B(R)$ and $\pi'=\B(R')$ differ by a single transposition of two consecutive elements $a$ and $b$.
By definition of $\B$ the adjacent transposition must correspond to an edge $e$ in $R$ that intersects the diagonal.
We need to show that $e$ is flippable. 
We proceed by contradiction, and suppose that $e$ is unflippable. 
Unflippable edges come in two flavors, and we consider the two cases separately.

In the first case, $e$ is unflippable because it is matched at both endpoints.
Suppose first that $e$ is horizontal and let $\ell_2=a$ and $\ell_3=b$ 
be the labels of the two rectangles above and below the diagonal, corresponding to the two labels involved in the transposition.
The two endpoints of $e$ must have their other incident edges oriented as shown on Figure~\ref{fig:proof1}.

Now consider the leftmost rectangle having his lower horizontal edge on the same segment as $e$, and denote its label by $\ell_1$ (this may be the rectangle that is just on the left of $\ell_2$, or another rectangle further left).
The diagonal must intersect rectangle $\ell_1$ before rectangle $\ell_2$.
Similarly, consider the rightmost rectangle having it upper horizontal edge on the same segment as $e$, and denote its label by $\ell_4$.
We must have $\ell_1<\ell_2<\ell_3<\ell_4$. 

In the $\boxslash$-order, the configuration of the rectangles is obtained by sliding the edges orthogonal to $e$ so that all the edges above $e$ are on the right of the edges below $e$.
By considering the four rectangles in the $\boxslash$-order, as illustrated on the right of Figure~\ref{fig:proof1}, we can check that $\pi$ contains the pattern $\ell_3\ldots \ell_4\ell_1\ldots \ell_2$.
Therefore, $\pi'$ contains the subsequence $\ell_2\ldots \ell_4\ell_1\ldots \ell_3$, an occurence of the forbidden pattern $2\underline{41}3$, and $\pi'$ cannot be a Baxter permutation, a contradiction.
A similar, symmetric, reasoning can be done when the unflippable edge is vertical, and then the forbidden pattern is $3\underline{14}2$.

In the second case, $e$ is matched at one endpoint only, but still unflippable because rotating it around its matched endpoint would yield a non-diagonal rectangulation.
Again, we have four symmetric cases, illustrated in Figure~\ref{fig:unflip2b}. 
We detail the case where $e$ is horizontal, and is matched at its left endpoint (top right case in the figure).
Figure~\ref{fig:proof2} illustrates what happens in this case.
Let $a=\ell_3$ and $b=\ell_4$ be the labels of the two rectangles above and below $e$, respectively.

From Lemma~\ref{lem:unflip2char}, the top horizontal edge of rectangle $\ell_3$ must be part of the obstruction to the rectangulation being diagonal.
Let us denote by $\ell_2$ the label of the predecessor of $\ell_3$ in the $\boxbackslash$-order.
Consider the rectangle labeled $\ell_1$ to the left of the left vertical edge of $\ell_3$. 
We clearly must have $\ell_1<\ell_2<\ell_3<\ell_4$.

Now remark that the top left corner of rectangle $\ell_4$ must be of type $\vdash$, since otherwise we would have a forbidden configuration for the diagonal representation. 
Therefore, in the $\boxslash$-order, no wall slide can be involved, and the rectangles $\ell_1,\ell_3$, and $\ell_4$ have the same relative positions.
By considering the four rectangles in the $\boxslash$-order, we can check that $\pi$ must contain the pattern $\ell_4\ldots \ell_1\ell_3\ldots \ell_2$ (see Figure~\ref{fig:proof2}).
By definition, $\pi'$ must contain the pattern $\ell_3\ldots \ell_1\ell_4\ldots \ell_2$, which is an instance of the forbidden pattern $3\underline{14}2$.
This is again a contradiction to the fact that $\pi'$ is a Baxter permutation.
The same reasoning can be done on the remaining three types of unflippable edges matched at one endpoint shown on Figure~\ref{fig:unflip2b}.
In all cases, we identify one of the two forbidden patterns in $\pi'$.

Therefore, the edge corresponding to the adjacent transposition in $\pi$ must be flippable, and the transposition indeed corresponds to a flip operation, as claimed.
\end{proof}

Hence this characterization of Barcelona flips is very similar to that of Law-Reading flips from Theorem~\ref{thm:LR}, except that the transpositions involve \emph{consecutive} elements instead of \emph{adjacent} elements, and that only a {\em single} transposition is needed.

\begin{figure}
\begin{center}
\includegraphics[page=8,scale=.6]{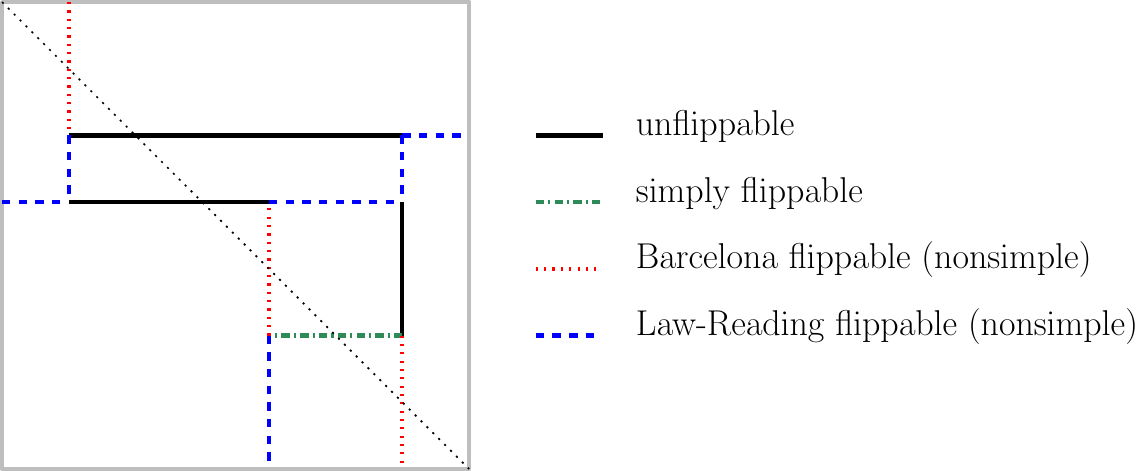}
\caption{\label{fig:all}The various types of flippable and unflippable edges.}
\end{center}
\end{figure}

\begin{figure}
\begin{center}
\includegraphics[height=.8\textheight]{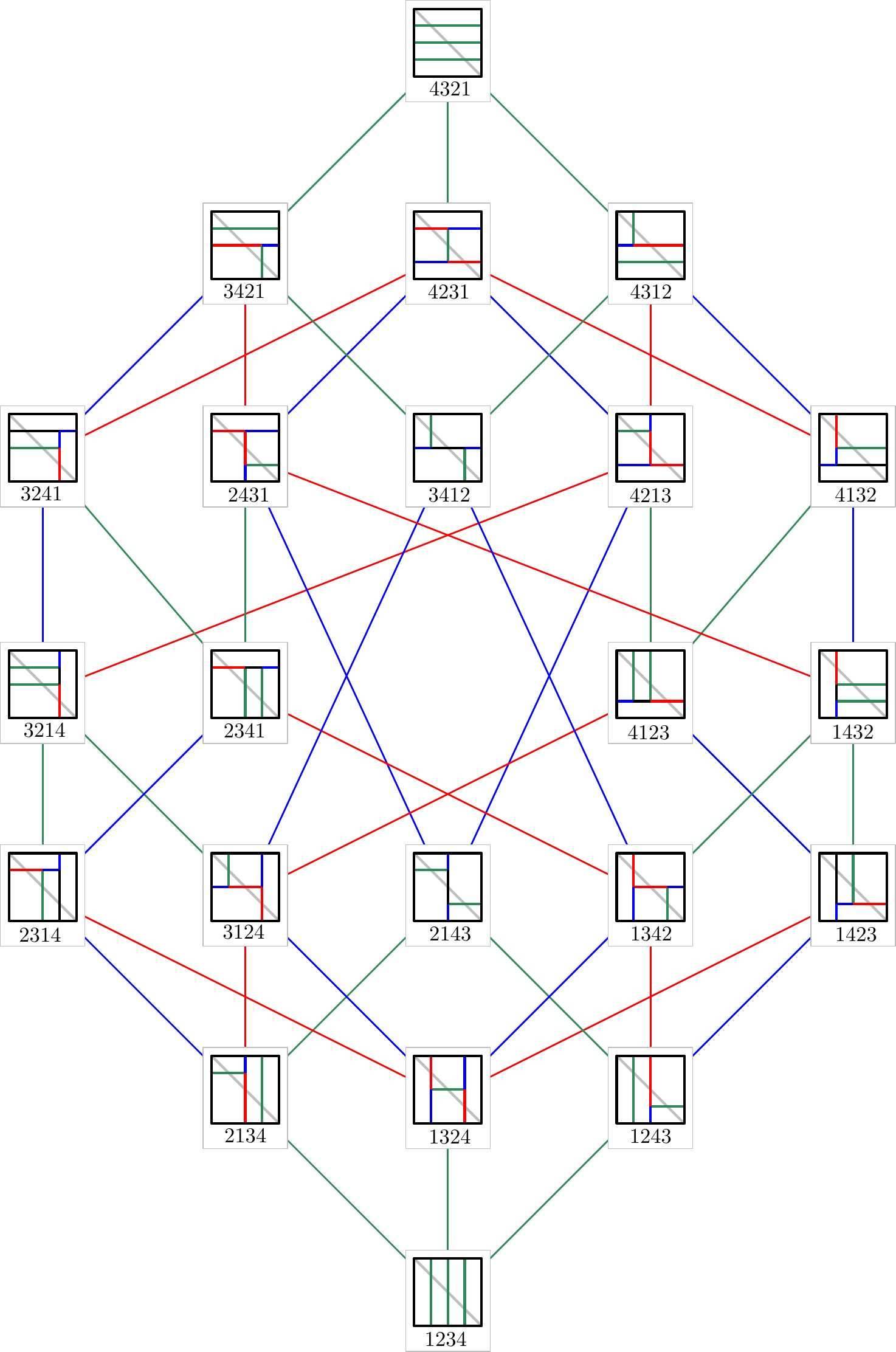}
\end{center}
\caption{\label{fig:flipgraph}The flip graph on diagonal rectangulations made of four rectangles. In each rectangulation, the green edges are simply-flippable, and the blue and red edges are respectively Law-Reading and Barcelona-flippable, but not simply-flippable. The links between the rectangulations are color-coded similarly. The Baxter permutation corresponding to each rectangulation is given.}
\end{figure}

\subsection{Characterization}

We can summarize our combinatorial characterization of flips in diagonal rectangulations as follows.
\begin{theorem}
Two diagonal rectangulations $R$ and $R'$ are connected by a flip if and only if one of these two conditions hold:
\begin{itemize}
\item $\B(R)$ and $\B(R')$ differ by a single transposition of consecutive elements,
\item $\B(R)$ and $\B(R')$ are in a cover relation in $\dRec_n$.
\end{itemize}
Furthermore, $R$ and $R'$ are connected by a simple flip if and only if both conditions hold.
\end{theorem}

Note that if the two permutations $B(R)$ and $B(R')$ differ by a single transposition of consecutive elements, their inverses $B(R)^{-1}$ and $B(R')^{-1}$ are in a cover relation in the weak order. The flip graph on diagonal rectangulations can therefore be seen as the union of the cover graph of $\dRec_n$ with those edges induced by the cover graph of the weak order on the inverse permutations. Figure~\ref{fig:all} shows all types of flippable and unflippable edges on our running example. The flip graph on diagonal rectangulations with four rectangles is given in Figure~\ref{fig:flipgraph}. 

\section*{Acknowledgments}
This work was initiated while the first author was visiting the UPC Research Group on Discrete, Combinatorial and Computational Geometry in March-April 2017. The authors wish to thank Andrei Asinowski, Stefan Felsner, and Vincent Pilaud for useful discussions and comments on a preliminary version of this manuscript. V.S. and R.I.S. were partially supported by projects Gen. Cat. DGR 2017SGR1640 and MTM2015-63791-R (MINECO/FEDER).  
R.I.S. was also supported by MINECO through the Ramón y Cajal program.

\bibliographystyle{plain}
\bibliography{flips}

\end{document}